\documentclass{article}
\usepackage{hyperref}
\usepackage{multirow}
\usepackage{graphicx} 
\usepackage{subfigure}
\usepackage[english]{babel}

\usepackage{pdflscape}
\usepackage{booktabs}
\usepackage{algorithm}
\usepackage{algorithmic}
\usepackage{comment}
\usepackage{amsthm}
\usepackage{latexsym, amsmath,amssymb, amsthm, mathtools}
\usepackage{amsmath, amsthm, amssymb}

\newtheorem{theorem}{Theorem}

\newtheorem{assumption}{Assumption}

\usepackage{multicol}

\usepackage{color}

\newcommand{\hatS}{\hat{\mathcal{S}}}
\newcommand{\hatSr}{\hat{\mathcal{S}}_{\rm row}}
\newcommand{\hatSc}{\hat{\mathcal{S}}_{\rm column}}

\newcommand{\sE}{\mathcal{E}}
\newcommand{\sL}{\mathcal{L}}
\newcommand{\sU}{\mathcal{U}}

\newcommand{\sA}{{\mathcal{L}}}
\newcommand{\sB}{{\mathcal{U}}}

\newcommand{\EER}{\mathcal{R}(r)}

\newcommand{\st}{\;:\;}

\newcommand{\R}{\mathbb{R}}
\newcommand{\Prob}{\mathbf{Prob}}

\begin{document}

\title{Matrix Completion under Interval Uncertainty}
\author{Jakub Mare\v{c}ek, Peter Richt\'arik, Martin Tak\'{a}\v{c}
\thanks{Jakub Mare\v{c}ek is at IBM Research, Martin Tak\'{a}\v{c} is at Lehigh University, and Peter Richt\'arik is at the University of Edinburgh.
Their addresses are \texttt{jakub@marecek.cz}, \texttt{takac.mt@gmail.com}, and
\texttt{peter.richtarik@ed.ac.uk},
respectively.
} }
\maketitle

\begin{abstract}
Matrix completion under interval uncertainty can be cast as matrix completion with element-wise box constraints.
We present an efficient alternating-direction parallel coordinate-descent method for the problem.
We show that the method outperforms any other known method on a benchmark in image in-painting in terms of signal-to-noise ratio,
and that it provides high-quality solutions for an instance of collaborative filtering with 100,198,805 recommendations 
within 5 minutes. 
\end{abstract}

\section{Motivation}

Matrix completion is a well-known problem, with applications ranging from image processing
to recommender systems.
When dimensions of a matrix $X$ and some of its elements $X_{i,j}, (i, j) \in \mathcal{E}$ are known,
 the goal is to find the unknown elements.
Without imposing any further requirements on $X$, there are infinitely many solutions.
In many applications, however, the matrix completion that minimizes the rank:
\begin{equation}\label{eq:basicNP_MC}
   \mbox{min}_Y \ \mbox{rank}(Y) \quad
   \mbox{subject to} \quad Y_{i,j} = X_{i,j},\quad (i,j)\in \mathcal{E},
\end{equation}
provides the simplest explanation for the data. 
There is a long history of work on the problem, c.f. \cite{Grigoriev84,sarwar2000,ye2005generalized,Koren2009},
with thousands of papers published annually since 2010.
 We hence cannot provide a complete overview. 


Let us note that 
Fazel \cite{fazel2002matrix} suggested to replace the rank, which is the sum of non-zero elements 
of the spectrum, with the nuclear norm, which is the sum of the spectrum.
The minimisation of the nuclear norm can be cast as a semidefinite programming (SDP) problem
and approaches based on the nuclear-norm have proven very successful in theory \cite{candes2009exact}
and very popular in practice. 
\cite{sarwar2000,cai2010singular} study the Singular Value Thresholding (SVT) algorithm.
This, however, required the computation of a singular value decomposition (SVD) in each iteration.
A number of other approaches, e.g., augmented Lagrangian methods \cite{tomioka2010fast}, appeared, but those would require a truncated SVD or
a number of iterations \cite{jaggi2010simple,Lee2010,shalev2011large,wang2014rank} of the power method.
Even considering the recent progress in randomized methods for approximating SVD, \cite{halko2011finding},
the approximation becomes very time-consuming as the dimensions of matrices grow.

A major computational break-through came in the form of the alternating least squares (ALS) algorithms \cite{srebro2004maximum,Rennie2005}. 
Initially, the algorithm has been used as a heuristic for finding stationary points
of the non-convex problem \cite{srebro2004maximum,Rennie2005,mnih2007probabilistic,4470228,4803763},
where a single iteration had complexity $O(|\mathcal{E}|r^2)$, for $|\mathcal{E}|$ 
observations and rank $r$, c.f., p. 60 in \cite{Keshavan2012}.
Keshavan et al.\ \cite{keshavan2010matrix,Keshavan2012}, however, proved its exponential
rate of convergence to the global optimum with high probability, under probabilistic
assumptions common in the compressed sensing community. 
Further, more technical analyses of the convergence to the global optimum 
have been performed by Jain et al. \cite{Jain2013}. 


Many studies of matrix completion consider the uncertainty, in some form. 
A number of analyses \cite{keshavan2010matrix,Keshavan2012,Jain2013} consider the
use of the standard rank-minimisation for the reconstruction of low-rank $m \times n$ matrix $XY^T$ from $XY^T + W$, where
$X \in \R^{m \times r}$, 
$Y \in \R^{n \times r}$,
$W \in \R^{m \times n}$
with elements of $W$ being bounded i.i.d. random variables, which are sub-Gaussian and 
have bounded expectation.
A number of further analyses \cite{NIPS20093704,candes2011robust} considered 
the use of the standard rank-minimisation for the reconstruction of low-rank $m \times n$ matrix
$XY^T$ from $XY^T + S$, where $X, Y$ are as above and $W$ has a small number of non-zero entries. 
\cite{Chen2011} consider some columns being corrupted. 
Although we are not aware of any studies of matrix completion under interval uncertainty, interval-based uncertainty has been considered in related problems. Alaiz et al.\ \cite{Alaiz2013} consider the min-max variant of the problem of finding the nearest correlation matrix, i.e., the
problem of finding the closest matrix within the set of symmetric positive definite matrices
with the unit diagonal to an uncertainty set, with respect to the Frobenius norm.
\cite{li2014robust} studied interval uncertainty in certain semidefinite programming problems,
which can be used to encode the nuclear-norm minimisation.

In contrast, we present an extension of matrix completion toward interval uncertainty,
 which has applications in image in-painting, collaborative filtering, and beyond. The algorithm we present for solving the problem can be seen as a coordinate-wise version of the ALS algorithm,
 which does not require the approximation of the spectrum of the matrix.
This makes it possible solve complete matrices $480,189 \times 17,770$ matrix within minutes on a standard laptop.
First, we provide an overview of the possible applications.


\subsection{Collaborative Filtering under Uncertainty}
Collaborative filtering  is a well-established application of matrix completion problems \cite{srebro2004learning},
 largely thanks to the success of the Netflix Prize.
There is a matrix, 
 where each row corresponds to one user and each column corresponds to a product or service.
Considering that every user rates only a modest number of products or services,
  there are only a small number of entries of the matrix known.
Our extension is motivated by the fact, that one user may provide two different ratings for one and the same product at two different times,
  depending on the current mood and other circumstances at the two times.
One may hence want to consider an interval $[\underline{x}, \overline{x}]$ instead of a fixed value $x$ of the rating,
  e.g., $[x - \epsilon, x + \epsilon]$.
Further, when one knows the scale $[0, M]$ the rating $x$ is chosen from, one can
 consider  $[\max\{ 0, x - \epsilon\}, \min \{ x + \epsilon, M\} ]$.
Hence, if intervals are known for elements $X_{i,j}$ of a matrix $X$ indexed by $(i,j)\in \mathcal{I}$, one may want to solve:
\begin{align}
   \mbox{min}_{Y_{i,j} \in [0, M]} \mbox{max}_{ X_{i,j} \in [\underline{X_{i,j}},  \overline{X_{i,j}}] \forall (i,j)\in \mathcal{I} } & \; \mbox{rank}(Y)   \label{eq:NetflixGeneral} \\ 
   \mbox{subject to} & \; Y_{i,j} = X_{i,j}, \quad \forall (i,j)\in \mathcal{I}. 
\notag
\end{align} 
Although numerous extensions of matrix completion problems have been studied, e.g. \cite{Mehta2007}, the use of robustness to interval uncertainty is novel.
It can be seen as an extension of robust optimisation \cite{Soyster1973} to matrix completion.

\subsection{Image In-Painting}
Further applications can be found in image processing.
In in-painting problems, a subset of pixels from an image are given and the goal is to fill in the missing pixels.
Rank-constrained matrix completion with equalities, where $\mathcal{I}$ is the index set of all known pixels, has been used numerous times 
\cite{candes2009exact,jain2010,mazumder2010spectral,goldfarb2009solving,Lee2010,jaggi2010simple,wang2014rank,wang2014rank} in this setting.
If the image comes from real sensors, it the corresponding matrix may have full (numerical) rank, but have quickly decreasing singular values in its spectrum.
In such a case, instead of solving the equality-constrained problem \eqref{eq:basicNP_MC},
one should like to find a low-rank approximation $Y^*$ of $X$,
such that the known entry of $X$ is not far away from $Y^*$, i.e.,
$\forall (i,j)\in \mathcal{I}$ we have $Y_{i,j} \approx X_{i,j}$.
Let us illustrate this with a small  matrix
\begin{equation}
X =\begin{pmatrix}
    68.16 &   78.12 &   24.04\\
   78.12  &   90.09 &    30.03\\
   24.04  &   30.03 &   20.01
     \end{pmatrix}, \notag
\end{equation}
which has rank 3 and its singular values $\Sigma=(167.9945, 10.2553,  0.0102)^T$.
It is easy to verify that
\begin{equation}
Y^*(2)=\begin{pmatrix}
68.1546 &  78.1250 &  24.0389 \\
78.1250 &  90.0853 &  30.0310 \\
24.0389 &  30.0310 &  20.0098
\end{pmatrix} \notag
\end{equation}
is the best rank 2 approximation of $X$ in   Frobenius norm.
Observe that no single element of $Y^*(2)$ is identical to $X$, but that
$Y^*(2)\approx X$. It is an easy exercise to show that
for any $X \in \R^{m \times n}$ with singular values
$\sigma_1 \geq \sigma_2\geq \dots\geq \sigma_{\min\{m, n\}}$,
and $Y^*(r)$ as its best rank-$r$ approximation, we have
$
 | X_{i,j} - (Y^*(r))_{i,j} | \leq \sum_{i=r+1}^{\min\{m, n\}}\sigma_{i} =: \EER$
for all $(i,j)$.
Therefore, one should not require equality constrains in \eqref{eq:basicNP_MC}, but rather inequalities
 $ | Y_{i,j} - X_{i,j}| \leq \EER, \forall (i,j)\in \mathcal{I}$.
Notice that this approach is not the same as minimizing $\sum_{(i,j)\in \mathcal{I}} (X_{i,j} - Y_{i,j})^2$
over all rank $r$ matrices, because we do not penalize the elements of $Y$, which are already close to $X$. It is also different from the usual treatment of noise
in the observations \cite{candes2010matrix}.
One could rather formulate this as the minimization of $\sum_{(i,j)\in \mathcal{I}} \max\{0, |X_{i,j} - Y_{i,j}|-\EER\}^2$ over all rank $r$ matrices.
Further, one knows the range of values allowed, e.g., $[0,1]$ for common encoding of gray-scale images.
This can hence be seen as ``side information'' which, as we will show in numerical section,
improves recovery of a low-rank approximation considerably.
Further still, one could assume that the intensity should be at least 0.8,
if pixels are missing within a light region of the image, or similar domain-specific heuristics.

\medskip
A number of other applications, e.g., in the recovery of structured matrices
\cite{chen2013spectral}, forecasting with side information, and in sparse principal component analysis with priors
on the principal components, can be envisioned.

\section{The Problem}

Formally, let $X$ be an $m \times n$ matrix to be reconstructed. Assume that elements
 $(i,j)\in \sE$ of $X$ we wish to fix, for elements $(i,j)\in \sL$ we have lower
 bounds and for elements $(i,j) \in \sU$ we have upper bounds. We employ the
 following natural formulation for the equality and inequality constrained
 matrix completion problem:
  \begin{align}
\label{eq:mainProblem}
\begin{split}
  \min_{X \in \mathbb \R^{m \times n}} \quad & \quad \mbox{rank}(X)\\
\mbox{subject to } \quad
      & X_{ij} =   X^{\sE}_{ij}, \ (i,j) \in \sE \\ 
      & X_{ij} \geq X^{\sL}_{ij}, \ (i,j) \in \sL \\
      & X_{ij} \leq X^{\sU}_{ij}, \ (i,j) \in \sU.\\
\end{split}
\end{align}

We shall enforce the following natural assumption:

\begin{assumption} $\sE \cap (\sL\cup \sU) = \emptyset$ and $X^{\sL}_{ij} \leq X^{\sU}_{ij} $ whenever $(ij)\in \sL\cap  \sU$.
\end{assumption}

The first condition says that if some element $(ij)$ is already fixed by an equality constraint, it does not (unnecessarily) appear any of the inequality constraints. The second condition says the upper and lower bounds should be consistent.

Problem~\eqref{eq:mainProblem} is NP-hard,
even with $\sB = \sA = \emptyset$ \cite{MR1320206,harvey2006complexity}.
A number of special cases of \eqref{eq:mainProblem} 
have been studied in the literature, e.g., in \cite{sarwar2000,niu2011hogwild,Kannan2012}.
A popular heuristic enforces low rank in a synthetic way by writing $X$ as a product of
two matrices, $X=L R$, where $L \in \R^{m \times r}$ and $R\in \R^{r \times n}$. Hence,
$X$ is of rank at most $r$ \cite{tanner2013normalized}.
Let $L_{i:}$ and $R_{:j}$ be the $i$-th row and $j$-th column of $L$ and $R$, respectively.
Instead of \eqref{eq:mainProblem}, we consider the \emph{smooth, non-convex}  problem
\begin{equation}\label{eq:NONCREF}
\min \{f(L,R)\;:\; L\in \R^{m\times r}, \; R\in \R^{r\times n}\},
\end{equation}
where
\begin{align}\label{defOff} 
f(L,R) :=  \tfrac{\mu}{2}\|L\|_{F}^2 &
+ \tfrac{\mu}{2}\|R\|_{F}^2 + f_{\sE}(L,R) + f_{\sL}(L,R) + f_{\sU}(L,R).
\end{align}
Above we have
\begin{eqnarray*}
f_{\sE}(L,R) &:=& \textstyle{\tfrac{1}{2}\sum_{(ij)\in \sE}(L_{i:}R_{:j}-X^{\sE}_{ij})^2}\\
f_{\sL}(L,R) &:=& \textstyle{\tfrac{1}{2}\sum_{(ij)\in \sL}(X^{\sL}_{ij}-L_{i:}R_{:j})_+^2}\\
f_{\sU}(L,R) &:=& \textstyle{\tfrac{1}{2}\sum_{(ij)\in \sU}(L_{i:}R_{:j}-X^{\sU}_{ij})_+^2},
\end{eqnarray*}
where
$\xi_+ = \max\{0,\xi\}$.

The parameter $\mu>0$ helps to prevent scaling issues\footnote{Let $X=L R$, then also $X=(cL)(\frac1c R)$ as well,
but we see that
for $c\to 0$ or $c\to \infty$ we have
$\|L\|_{F}^2 + \|R\|_{F}^2 \ll \|cL\|_{F}^2 + \|\frac1c R\|_{F}^2$.
}. We could optionally set $\mu$ to zero and instead, from time to time, rescale matrices $L$ and $R$, so that their product is not changed \cite{tanner2013normalized}. The term $f_\sE$ (resp.\ $f_\sU$, $f_\sL$) encourages  the equality (resp.\ inequality) constraints to hold.

\section{The Method}

  \begin{algorithm}[t]
 \caption{MACO: Matrix Completion via Alternating Parallel Coordinate Descent}
 \label{alg:SCDM}
 \begin{algorithmic}[1]
 \item[] Input: $\sE, \sL, \sU, X^{\sE}, X^{\sL}, X^{\sU}$, rank $r$
  \item[] Output: $m \times n$ matrix $LR$ 
 \STATE \label{alg:stp:initialpoint} choose $L\in \R^{m\times r}$ and $R\in \R^{r\times n}$
 \FOR{$k=0,1,2,\dots$}
   \STATE choose random subset $\hatSr \subset \{1,\dots,m\}$
   \FOR{$i\in \hatSr$ {\bf in parallel}}
      \STATE choose $\hat r \in \{1,\dots,r\}$ uniformly at random
      \STATE compute $\delta_{i \hat r}$ using formula \eqref{eq:delta_L}
     \STATE update $L_{i \hat r}  \leftarrow  L_{i \hat r} + \delta_{i \hat r}$
   \ENDFOR
   \STATE choose random subset $\hatSc \subset \{1,\dots,n\}$
      \FOR{$j\in \hatSc$ {\bf in parallel}}
      \STATE choose $\hat r \in \{1,\dots,r\}$ uniformly at random
	\STATE compute $\delta_{\hat{r}j}$ using \eqref{eq:delta_R}
	\STATE update $R_{\hat{r}j} \leftarrow R_{\hat{r}j} + \delta_{\hat{r} j}$
   \ENDFOR
 \ENDFOR
 \end{algorithmic}
\end{algorithm}

Coordinate descent algorithms (CDA) are effective in solving large-scale problems, due to their low per-iteration computational cost.
Although each iteration of CDA is cheap, many more iterations are required for convergence, compared to second-order algorithms or similar.
Recently, the stochastic CDA has received much attention
 \cite{nesterov2012efficiency,RT-serial}
 not least due to the parallelizability \cite{PCDM, NSync, DQA,
ImprovedPCDM} with almost linear speed-up in regimes with sparse data, when the number of parallel updates $\tau$ is much smaller that the dimension of the optimization problem \cite{niu2011hogwild}. Distributed variants have also been studied \cite{Hydra3,Hydra}. 

In Algorithm \ref{alg:SCDM},
we present our alternating parallel coordinate descent method for MAtrix COmpletion, henceforth simply ``MACO''. In Steps 3--8 of our algorithm, we fix $R$, choose random $\hat{r}$ and a random set $\hatSr$ of rows of $L$, and
update, in parallel for $i \in \hatSr$:  $L_{i\hat{r}} \leftarrow L_{i\hat{r}} + \delta_{i\hat{r}}$.
In Steps 9--14, we fix $L$, choose random $\hat{r}$ and a random set $\hatSc$ of columns of $R$, and update, in parallel
for $j \in \hatSc$: $R_{\hat{r}j}\leftarrow R_{\hat{r}j} + \delta_{\hat{r}j}$.

Let us now comment on the computation of the updates, $\delta_{i\hat{r}}$ and $\delta_{\hat{r}j}$. First, note that while $f$ is not convex jointly
in $(L,R)$, it is convex in $L$ for fixed $R$ and in $L$ for fixed $R$.

\subsection{Row Update}

If we now fix row $i\in \{1,2,\dots,m\}$ and $\hat{r}\in \{1,2,\dots,r\}$, and view $f$ as a function of $L_{i\hat{r}}$ only, it has a Lipschitz continuous derivative with constant
\begin{equation}\label{eq:asfdsafsa}
W_{i\hat{r}} =
 W_{i\hat{r}}(R):= \mu + \sum_{j \st(i j) \in \sE}
        R_{\hat{r}j}^2
       + \sum_{j \st(i j) \in \sL\cup\sU}
        R_{\hat{r}j} ^2.
\end{equation} 
That is, for all $L$, $R$ and $\delta\in \R$, we have
\begin{equation}\label{eq:iuhs98y9s8}f(L+\delta E_{i\hat{r}},R) \leq f(L,R) + \langle \nabla_L f(L,R), E_{i\hat{r}}\rangle \delta  + \frac{W_{i\hat{r}}}{2}\delta^2,\end{equation}
where $E_{i\hat{r}}$ is the $n\times r$ matrix with $1$ in the $(i \hat{r})$ entry and zeros elsewhere.  The minimizer of the right hand side of \eqref{eq:iuhs98y9s8} in $\delta$ is given by
\begin{equation} \label{eq:delta_L}\delta_{i\hat{r}}:= - \langle \nabla_L f(L,R), E_{i\hat{r}}\rangle / W_{i\hat{r}},\end{equation}
where \begin{eqnarray*}\langle \nabla_L f(L,R), E_{i\hat{r}}\rangle = 
\mu L_{i \hat{r}} &+& \sum_{j \st (i j) \in \sE}
                ( L_{i:} R_{:j} - X^\sE_{ij}) R_{\hat{r}j }  \\
&+&\sum_{j \st (i j) \in \sU \;\&\; L_{i:} R_{:j} < X_{ij}^\sU }   ( L_{i:} R_{:j}-X_{ij}^\sU) R_{\hat{r}j}\\
&+&\sum_{j \st (i j) \in \sL \;\&\; L_{i:} R_{:j} > X_{ij}^\sL} (L_{i:} R_{:j}-X_{ij}^\sL) R_{\hat{r}j}.
\end{eqnarray*} Note that
\begin{equation}\label{eq:sjsus8}f(L+\delta_{i\hat{r}} E_{i\hat{r}}, R) \leq f(L,R) - \frac{\langle \nabla_L f(L,R),E_{i\hat{r}}\rangle^2}{2W_{i\hat{r}}}.\end{equation}
Let $W^{(k)}_{i\hat{r}} := W_{i\hat{r}}(R^{(k)})$ be the value of the Lipschitz  constant at iteration $k$. 

\subsection{Column Update}

Likewise, if we now fix $\hat{r}\in \{1,2,\dots,r\}$ and column $j\in \{1,2,\dots,n\}$, and view $f$ as a function of $R_{\hat{r}j}$ only, it has a Lipschitz continuous derivative with constant
\[V_{\hat{r}j} = V_{\hat{r}j}(L)
 := \mu + \sum_{i \st (i j) \in \sE}
        L_{i\hat{r}}^2
       + \sum_{i \st(ij) \in \sU\cup\sL}
        L_{i\hat{r}}^2.\]
That is, for all $L$, $R$ and $\delta\in \R$, 
\begin{equation}\label{eq:iugs9t98ed}f(L,R+\delta E_{\hat{r}j}) \leq f(L,R) + \langle \nabla_R f(L,R), E_{\hat{r}j}\rangle \delta  + \frac{V_{\hat{r}j}}{2}\delta^2,\end{equation}
where $E_{\hat{r}j}$ is the $r\times m$ matrix with $1$ in the $(\hat{r}j)$ entry and zeros elsewhere. The minimizer of the right hand side of \eqref{eq:iugs9t98ed} in $\delta$ is given by
\begin{equation} \label{eq:delta_R}\delta_{\hat{r}j}:= - \langle \nabla_R f(L,R), E_{\hat{r}j}\rangle / V_{\hat{r}j},\end{equation}
where 
\begin{eqnarray*}
\langle \nabla_R f(L,R), E_{\hat{r}j}\rangle =
\mu R_{\hat{r}j} &+& \sum_{i \st (ij) \in \sE}
                ( L_{i:} R_{:j} - X^\sE_{ij}) L_{i\hat{r}} \\
&+& \sum_{i \st (ij) \in \sL \;\&\; L_{i:} R_{:j} < X_{ij}^\sL }
                 ( L_{i:} R_{:j}-X_{ij}^\sL) L_{i\hat{r}} \\
&+& \sum_{i \st (ij) \in \sU \;\&\; L_{i:} R_{:j} > X_{ij}^\sU }
                 (L_{i:} R_{:j}-X_{ij}^\sU) L_{i\hat{r}}.
\end{eqnarray*}
Note that
\begin{equation}\label{eq:sjsusss8}f(L, R+\delta_{\hat{r}j} E_{\hat{r}j}) \leq f(L,R) - \frac{\langle \nabla_R f(L,R),E_{\hat{r}j}\rangle^2}{2V_{\hat{r}j}}.\end{equation}
Let $V^{(k)}_{\hat{r}j} := W_{\hat{r}j}(L^{(k)})$ be the value of the Lipschitz  constant at iteration $k$. 

\subsection{Row and Column Sampling}

The random set (``sampling'') $\hatSr$ defined in Step 3 (resp sampling $\hatSc$ in Step 10) can have an arbitrary distribution as long as it contains every row (resp column) of matrix $L$ (resp $R$) with positive probability. We shall now formalize this.

\begin{assumption} The samplings $\hatSr$ and $\hatSc$ are  {\em proper}, i.e., 
\[\Prob(i\in \hatSr) > 0 \quad \text{for all} \quad i\in \{1,2,\dots,m\},\]
and
\[\Prob(j\in \hatSc) > 0 \quad \text{for all} \quad j\in \{1,2,\dots,n\}.\]
\end{assumption}

In particular, we can  chose the random sets $\hatSr$ (resp $\hatSc$) so that every row (resp column) has equal probability of being chosen. Samplings with this property are called {\em uniform}, and we use this choice in our experiments. However, our theory also allows for nonuniform samplings. If we have a multicore machine available with $\tau$ cores, then a reasonable sampling  should have cardinality $\tau$, or some integral multiple of $\tau$, so that every core has a reasonable (not too small to be underutilized, but not too large either, so as to avoid long processing time) load at every iteration.

\subsection{The Final Step}

Formulae \eqref{eq:delta_L} and \eqref{eq:delta_R}
suggest that the computation of the final step is very computationally demanding. This can, however, be avoided if we define matrices $A\in\R^{m\times r}$ and $B\in\R^{r\times n}$ such that $A_{iv}=W_{iv}$ and $B_{vj}=V_{vj}$. After each update of the solution, we can also update those matrices. Similarly, one can store sparse  residuals matrices
$\Delta_\sE$, $\Delta_\sL$, $\Delta_\sU$, where
$$(\Delta_\sE)_{i,j}=\begin{cases}
                    L_{i:} R_{:j} -X^\sE_{ij},&\mbox{if}\ (ij)\in\sE
\\
                     0,&\mbox{otherwise,}
                    \end{cases}$$
and $\Delta_\sU$, $\Delta_\sL$ are defined in similar way.
Subsequently, the computation of $\delta_{i \hat r}$ or $\delta_{ \hat{ r} j}$
is reduced to just a few multiplications and additions.

\subsection{Convergence Analysis}

 Due to the non-convex nature of \eqref{eq:NONCREF}, one 
 has to be satisfied with convergence to a stationary point, in general.

\begin{theorem}
Let $\mu>0$ and 
and let $(L^{(k)}, R^{(k)})$ be the (random) matrices produced by Algorithm~\ref{alg:SCDM} after $k$ iterations, assuming that $\hatSr$ and $\hatSc$ are proper. Then  for all $k\geq 0$,
\begin{equation}\label{eq:monotonicity}
  0 \leq  f(L^{(k+1)},R^{(k+1)}) \leq f(L^{(k)},R^{(k)}).
\end{equation}
That is, the method is monotonic. Moreover, with probability 1,
\[\lim_{k\to \infty} \inf \|\nabla_L f(L^{(k)},R^{(k)})\|
= 0,\] 
and  
\[\lim_{k\to \infty} \inf \|\nabla_R 
f(L^{(k)},R^{(k)})\| = 0.\]
\end{theorem}
\begin{proof}
From 
\eqref{eq:sjsus8}
and \eqref{eq:sjsusss8}
we see that for all $k$ we have
$$
f(L^{(k)},R^{(k)})
\overset{\eqref{eq:sjsus8}}{\geq} 
f(L^{(k+1)},R^{(k)})
\overset{\eqref{eq:sjsusss8}}{\geq} 
f(L^{(k+1)},R^{(k+1)}) \geq 0,
$$ 
where the last inequality follows from the fact that all parts of $f$ defined in \eqref{defOff} are non-negative. 

Monotonicity \eqref{eq:monotonicity} together with the fact that  $\mu>0$
imply that the level set \[\Omega_0:=\{ (L,R) \st f(L,R) \leq f(L^{(0)},R^{(0)})\} \] is bounded. Now, for all $i\in \{1,2,\dots,m\}$, $v\in \{1,2,\dots,r\}$ and any iteration counter $k$ we have
\begin{equation}\label{eq:oi98s98g}\mu \overset{\eqref{eq:asfdsafsa}}{\leq} W^{(k)}_{iv} \overset{\eqref{eq:asfdsafsa}}{\leq}  \mu + \|R^{(k)}\|_F^2 \overset{\eqref{defOff} }{\leq} \mu + \frac{2}{\mu} f(L^{(k)},R^{(k)}) \leq \mu + \frac{2}{\mu}f(L^{(0)},R^{(0)}).\end{equation}
In the second inequality we have used Assumption 1, and in the last inequality we have used monotonicity. The same lower and upper bounds can be established for $V_{vj}^{(k)}$.

We shall now establish that $\lim \inf \|\nabla_L f(L^{(k)},R^{(k)})\|_F^2= 0$ with probability~1 (the claim  $\lim \inf \|\nabla_R f(L^{(k)},R^{(k)})\|_F^2=0$ can be proved in an analogous way). Since
\[\|\nabla_L f(L^{(k)},R^{(k)})\|_F^2 = \sum_{i=1}^m \sum_{v=1}^r  \langle \nabla_L f(L^{(k)},R^{(k)}), E_{iv}\rangle^2,\]
 it is enough to show that  for
$\Delta_{iv}^{(k)}:=\langle \nabla_L f(L^{(k)},R^{(k)}), E_{iv}\rangle$ we have $\lim \inf (\Delta_{iv}^{(k)})^2 = 0$  with probability 1 for  all $i \in \{1,2,\dots,m\}$ and $v\in \{1,2,\dots,r\}$. Fix any $i$ and $v$. Since $\hatSr$ is proper, and since $\hat{r}$ is chosen uniformly at random in each iteration, there is an infinite sequence of iterations, indexed by $K_{iv}$,  in which the pair $(i,v)$ is sampled.

In view of \eqref{eq:sjsus8} and \eqref{eq:oi98s98g}, for all $k\in K_{iv}$ we have
$$
f(L^{(k+1)},R^{(k+1)}) \leq f(L^{(k+1)},R^{(k)})  \leq f(L^{(k)},R^{(k)}) - \frac{(\Delta_{iv}^{(k)})^2}{C},
$$
where $C=2(\mu + \frac{2}{\mu}f(L^{(0)},R^{(0)}))$.
Since $f(L,U)$ is nonnegative, it must be the case that
$\sum_{k\in K_{iv}} (\Delta_{iv}^{(k)})^2$
is finite. 
This means that, with probability 1, $\lim_{k\to \infty} \inf (\Delta_{iv}^{(k)})^2 = 0$, as desired.


\end{proof}

\section{Computational Results and a Discussion}

We have conducted a variety of experiments. First, we present 
the performance in collaborative filtering,
next we compare the performance in image in-painting with classical
matrix completion techniques with $\sU\equiv \sL \equiv \emptyset$.
We conclude with remarks on the run-time and hardware used.

\subsection{Collaborative Filtering}

In our computational testing of collaborative filtering, we start with  \verb!smallnetflix_mm!,
where the training dataset contains $c_{\mbox{tr}}=3,298,163$ integers out of $\{1,2,3,4,5\}$,
which describe how $m=95,526$ users rate $n=3,561$ movies.
Second, we use a well-known data-set, which contains $100,198,805$ ratings on the same scale,
obtained from $480,189$ users considering $17,770$ products,
as available from CMU\footnote{ \tiny \url{http://www.select.cs.cmu.edu/code/graphlab/datasets/}
}.
Third, we use
Yelp's Academic Dataset\footnote{ \tiny \url{https://www.yelp.co.uk/academic_dataset}},
from which we have extracted a $252,898 \times 41,958$ matrix with 1,125,458 non-zeros,
 again on the 1--5 scale.

Although we know some ratings exactly on \verb!smallnetflix_mm!, we consider \eqref{eq:NONCREF} 
of \eqref{eq:mainProblem} with interval uncertainty sets of width 2:
\begin{equation}
 \begin{split}
Y_{i,j} \leq \min\{5, X_{i,j}+1\}, (i,j)\in \mathcal{I},\\
Y_{i,j} \geq \max\{1, X_{i,j}-1\}, (i,j)\in \mathcal{I}.
  \end{split}
  \label{eq:Netflix}
\end{equation}
In particular, we 
complete a $95526 \times 3561$ matrix of rank 2 or 3,
possibly using width-2 interval uncertainty set and
scale of 1 to 5 stars in the ratings.
To illustrate the impact of the this change, 
 we present the evolution of
 Root-Mean-Square Error (RMSE)
in Figure~\ref{fig:RMSE} (left). Notice that an ``epoch'', which is the unit on the horizontal axis, consists of $c_{\mbox{tr}}$ element updates of matrix $L$ and $c_{\mbox{tr}}$ element updates of matrix $R$.

Let us remark that RMSE is sensitive to the choice of  $\Delta$ and the rank of the matrix we are looking for.
If the underlying matrix has a higher rank than expected, $\Delta>0$ can lead to smaller values of RMSE.
We should also note that for some fixed $\Delta_1$ and $\Delta_2$, RMSE can be better with $\Delta_1$ for a few epochs, but then get worse when compared with  $\Delta_2$. 
Hence, in practice, cross validation should be used to determine suitable value of parameter $\Delta$.

On the 
Yelp data set, we have performed 10-fold cross-validation on the training set, using varying rank.
As we increased the rank from 1 to 2, 4, 8, 16, 32, and 50, the average error decreased
from 1.7958 to 1.8284, 1.6464, 1.4590, 1.3395, 1.2702, and 1.2454, respectively.
This seems to be comparable to the best results 
from the  2013 Recommender Systems Challenge\footnote{
\tiny \url{https://www.kaggle.com/c/yelp-recsys-2013} 
}, 
where a smaller dataset was used.

Further, one can illustrate the effects in a matrix-recovery experiment. 
We use random matrices $X\in\R^{20 \times 20}$ of rank $8$.
We sample $p\%$ of entries of the matrix and store their indices in $\mathcal{I}$.
We solve \eqref{eq:NONCREF} with just the inequality constrains, i.e.,
$\sE \equiv \emptyset, \sB \equiv \sA \equiv \mathcal{I}$,
$X^\sB = X - \Delta {\bf 1}$ and $X^\sA = X + \Delta {\bf 1}$, where ${\bf 1}\in \R^{m \times n}$ is a matrix with all elements equals to $1$.
Let us denote by $Y^*(\Delta)$ the solution of that optimization problem after $10^5$ serial iterations ($|\hatS|=1$) and with $\mu=10^{-5}$.
Figure \ref{fig:EXP1} shows the dependence of error defined as follows
$
 \mbox{Error}(\Delta) = \frac{\| Y^*(\Delta) - X(7) \|_F}{\| X(7) \|_F},$ where $X(r)$ is the best rank $r$ approximation of $X$ obtain using SVD decomposition of the whole matrix.
Figure \ref{fig:EXP1} clearly suggest that, e.g., if $50\%$ of elements are observed then by allowing each entry $\in \mathcal{I}$ of reconstructed matrix
to lie in $\Delta$ neighborhood of observed values, we can decrease the relative error of reconstruction from approximately $1.22$ to $0.4$ for $\Delta \approx \EER$. In this case, the value of $\|X(7)\|_F$ was  $21.3245$ and   $\EER=0.1075$.

\begin{figure*}[tp!]
 \centering
 \begin{tabular}{c c c}
  $p=30$ & $p=50$ & $p=80$ \\

 \includegraphics[width=0.3\textwidth]{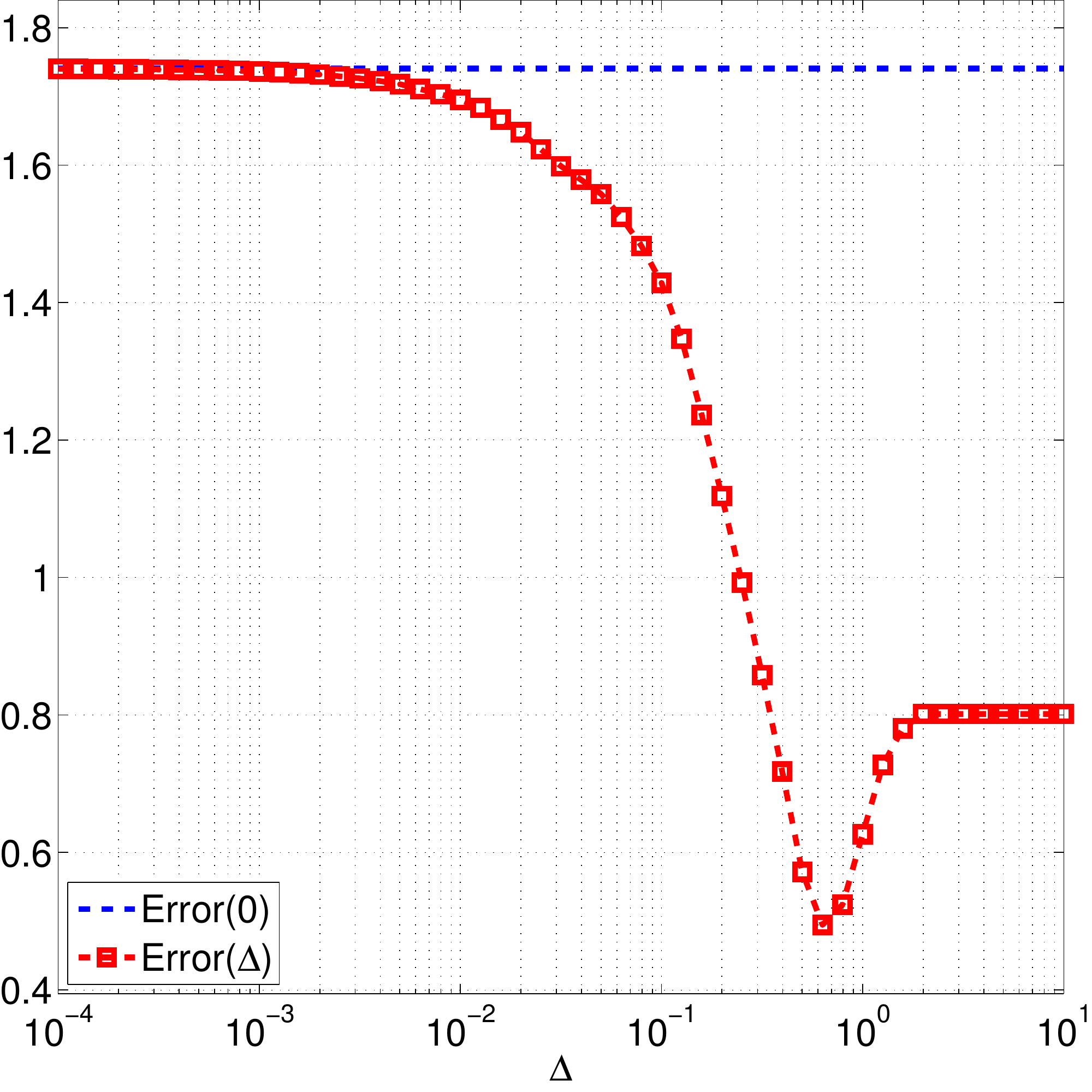} &
 \includegraphics[width=0.3\textwidth]{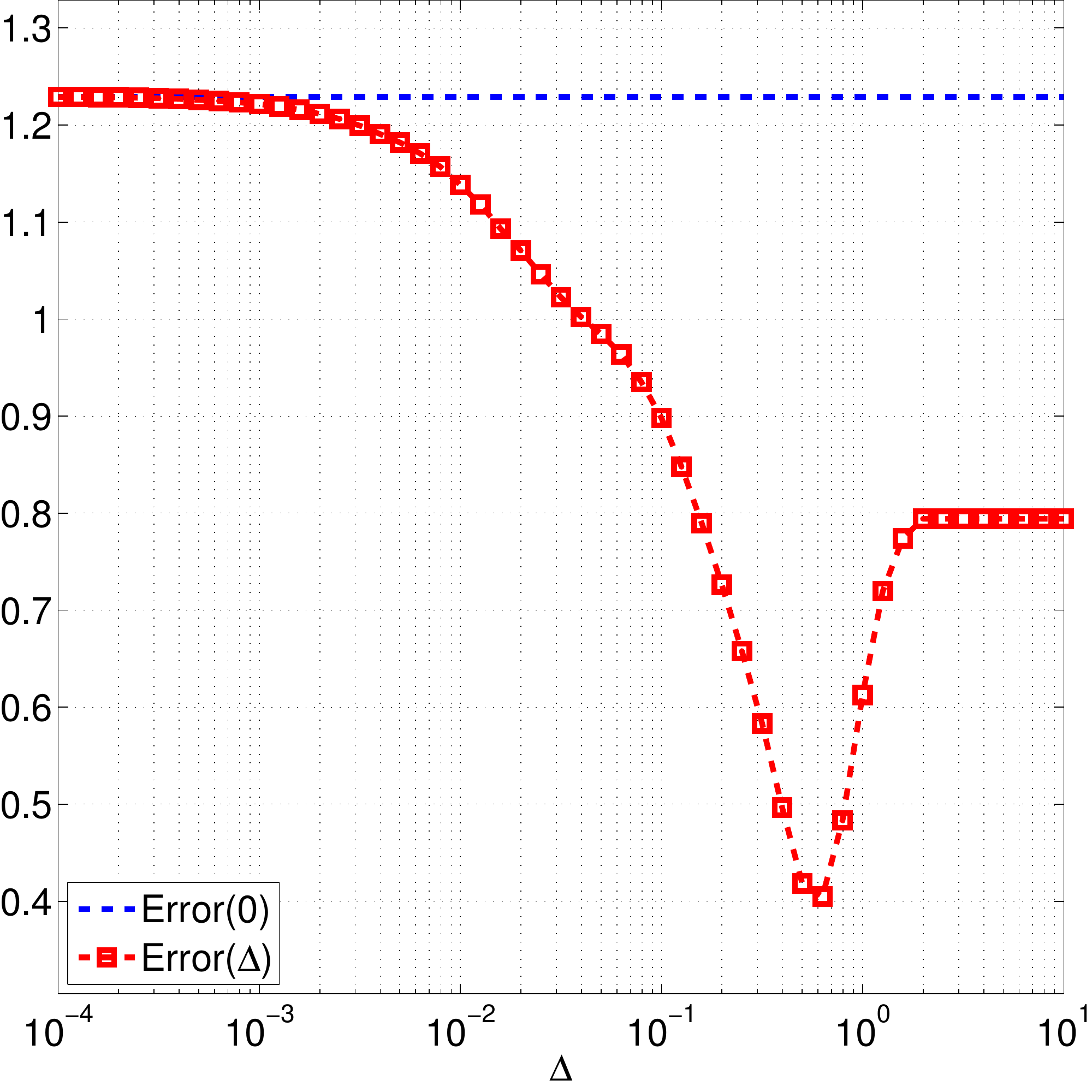} &
 \includegraphics[width=0.3\textwidth]{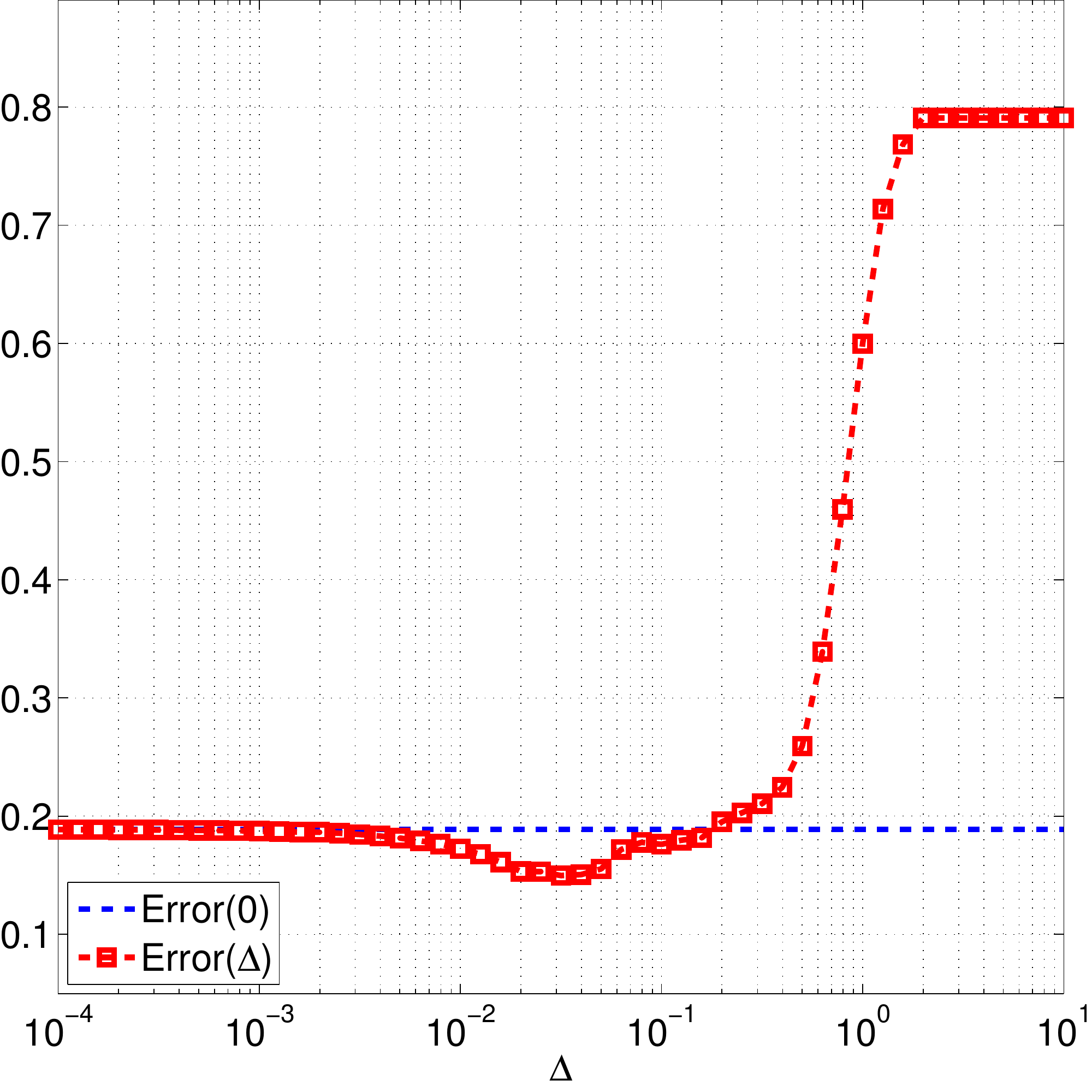}
 \end{tabular}

 \caption{Dependence of $\mbox{Error}$ on $\Delta$ for various $p \in \{30, 50, 80\}$ in matrix reconstruction.}
 \label{fig:EXP1}
\end{figure*}

\begin{figure*}[tp!]
\centering
\includegraphics[width=0.3\textwidth,clip,trim=0mm 0mm 0mm 0mm]{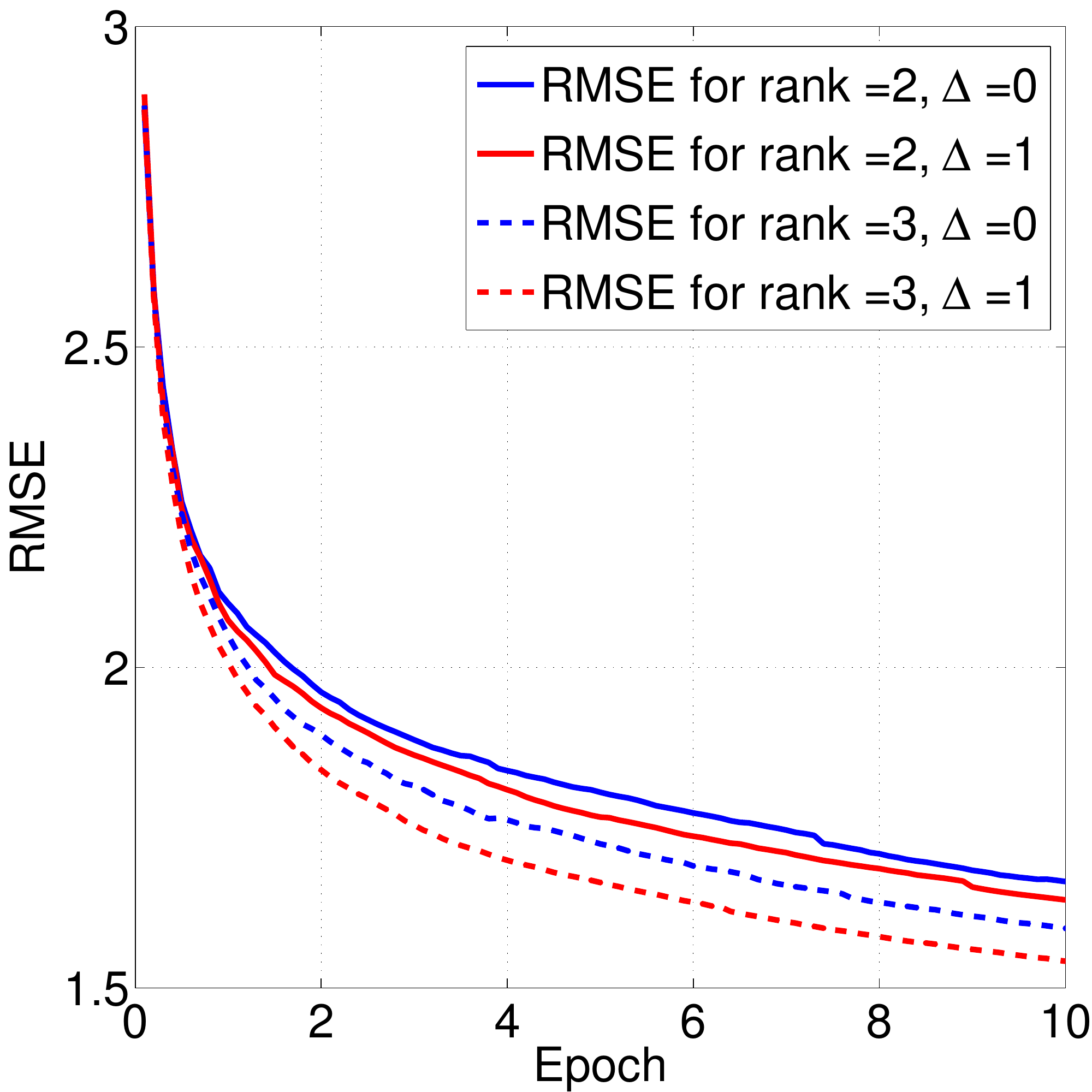}
\includegraphics[width=0.3\textwidth,clip,trim=0 0 0 7cm]{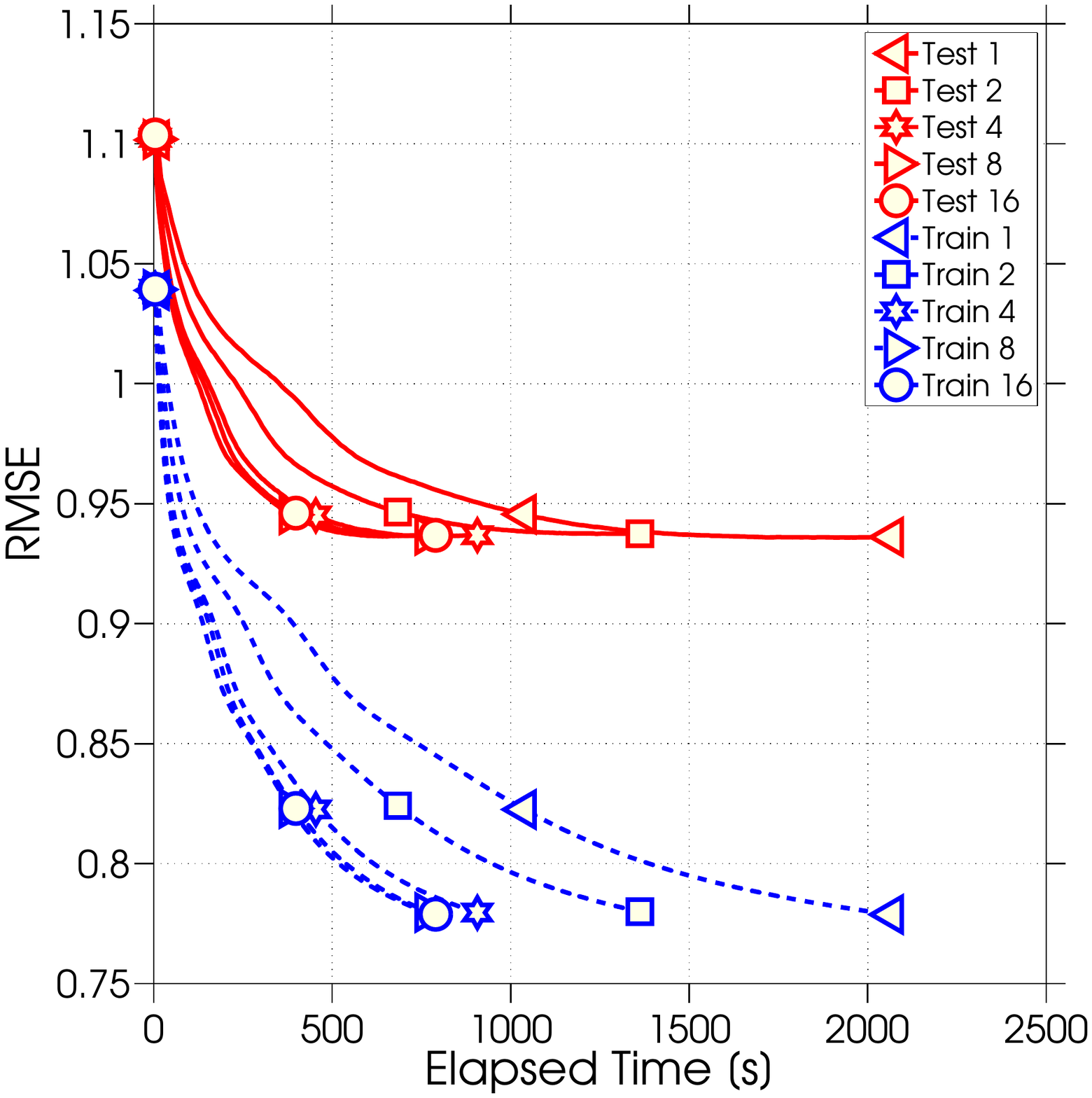}
\includegraphics[width=0.3\textwidth,clip,trim=0 0 0 7cm]{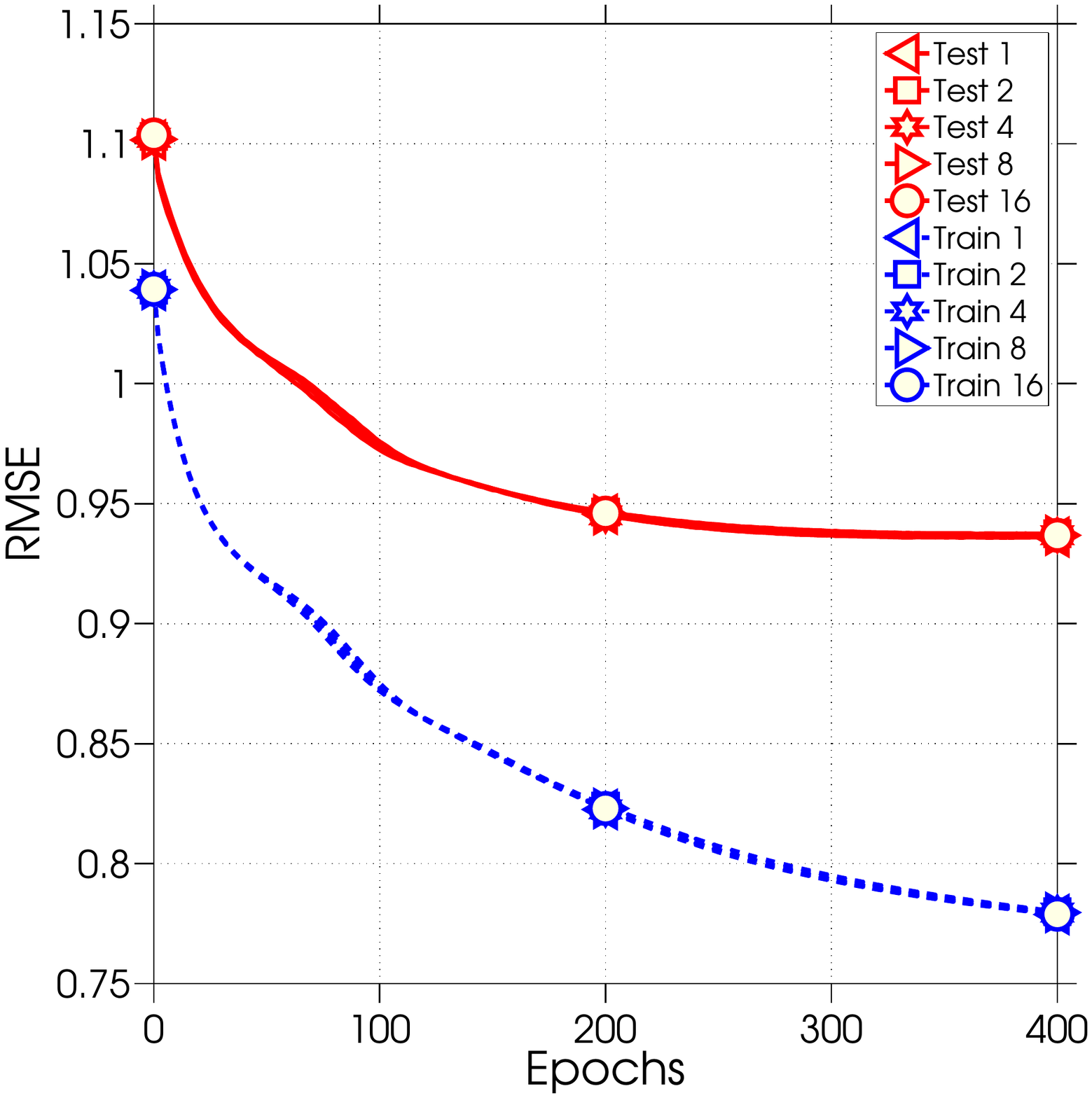}
\caption{
Left: The effect of adding inequalities ($\Delta = 1$) to the equality-constrained problem ($\Delta = 0$)
on {\tt smallnetflix}, for $r = 2, 3$, $\mu=10^{-3}$.
Center and right: 
RMSE as a function of the number of iterations and wall-clock time, respectively, 
on a well-known $480189 \times 17770$ matrix, for $r=20$ and $\mu=16$.
}
\label{fig:RMSE}
\end{figure*}

\subsection{Image In-Painting}

Further, we provide a comparison on the in-painting benchmark of \cite{wang2014rank}. Table~\ref{table:ir} details the performance of 
SVT \cite{candes2009exact}, SVP \cite{jain2010}, SoftImpute \cite{mazumder2010spectral}, LMaFit \cite{goldfarb2009solving}, ADMiRA, \cite{Lee2010}, JS \cite{jaggi2010simple}, 
OR1MP \cite{wang2014rank}, and EOR1MP \cite{wang2014rank} on 
10 well-known gray-scale images 
(Barbara, Cameraman, Clown, Couple, Crowd, Girl, Goldhill, Lenna, Man, Peppers) of $512 \times 512$ pixels each. 
$50\%$ of pixels were removed uniformly at random, 
and the image was reconstructed using rank 50. 
The performance was measured in terms of PSNR, which is $10 \log_{10}(255^2 / E)$ for mean squared error $E$.
Our approach with inequalities $0 \le Y_{i,j} \le 255$ dominates 
all other approaches on 7 out of the 10 images. 
On the remaining 3 images, one would have to use the extrema of the 
observed elements, e.g., a subinterval of 12--246 for Barbara.

To illustrate the aggregate results further, we undertook the following experiment.
We took a $512\times 512$ gray scale image (Lenna) and chose $50\%$ of the pixels randomly, indexed as $\mathcal{I}$.
Then, we ran Algorithm \ref{alg:SCDM} for $10^7$ serial iterations ($|\hatS|=1$).
We obtained
solutions $X_{E}(\mbox{rank})$ and $X_{IN}(\mbox{rank})$, where
$X_{E}(\mbox{rank})$ was obtained when we used only equality constrains ($\sE=\mathcal{I}, \sB\equiv \sA \equiv \emptyset$)
and
$X_{IN}(\mbox{rank})$ was obtained when we used also inequality constrains
($\sE=\mathcal{I}$, $\sU\equiv \sL \equiv -\mathcal{I}$,
$X^\sL = {\bf 0}\in\R^{512 \times 512}$, $X^\sU = {\bf 1}\in\R^{512 \times 512}$ and $-\mathcal{I}$ is a set of all elements of $X$ except those in $\mathcal{I}$).
Figure \ref{fig:EXP3} shows for different $\mbox{rank}\in\{30,50,100\}$ the best $\mbox{rank}$ approximation obtained by SVD ($X(\mbox{rank})$)
and solutions $X_{E}(\mbox{rank})$ and $X_{IN}(\mbox{rank})$.
The benefit of obvious inequality constrains is nicely visible, e.g., at $\mbox{rank}=100$, where the
relative error of reconstruction is more than twice smaller. Further, the image is more smooth, upon visual inspection.

\begin{figure*}[tp!]
 \centering \small
 \begin{tabular}{c|c|c|c}
 rank & $X(\mbox{rank})$ & $X_{E}(\mbox{rank})$ & $X_{IN}(\mbox{rank})$
 \\ \hline \hline
30 &  \includegraphics[width=1.5in]{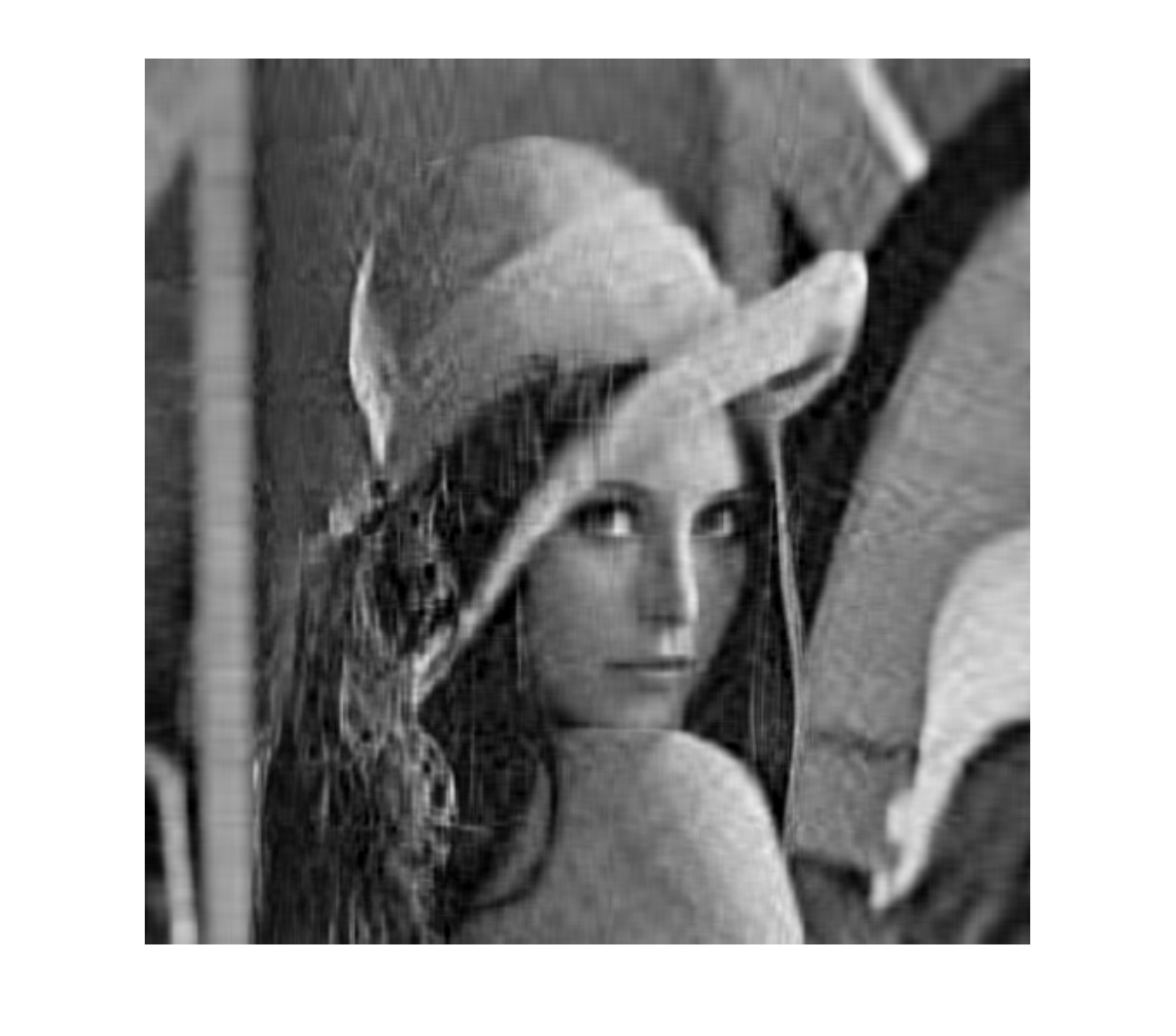} &
  \includegraphics[width=1.5in]{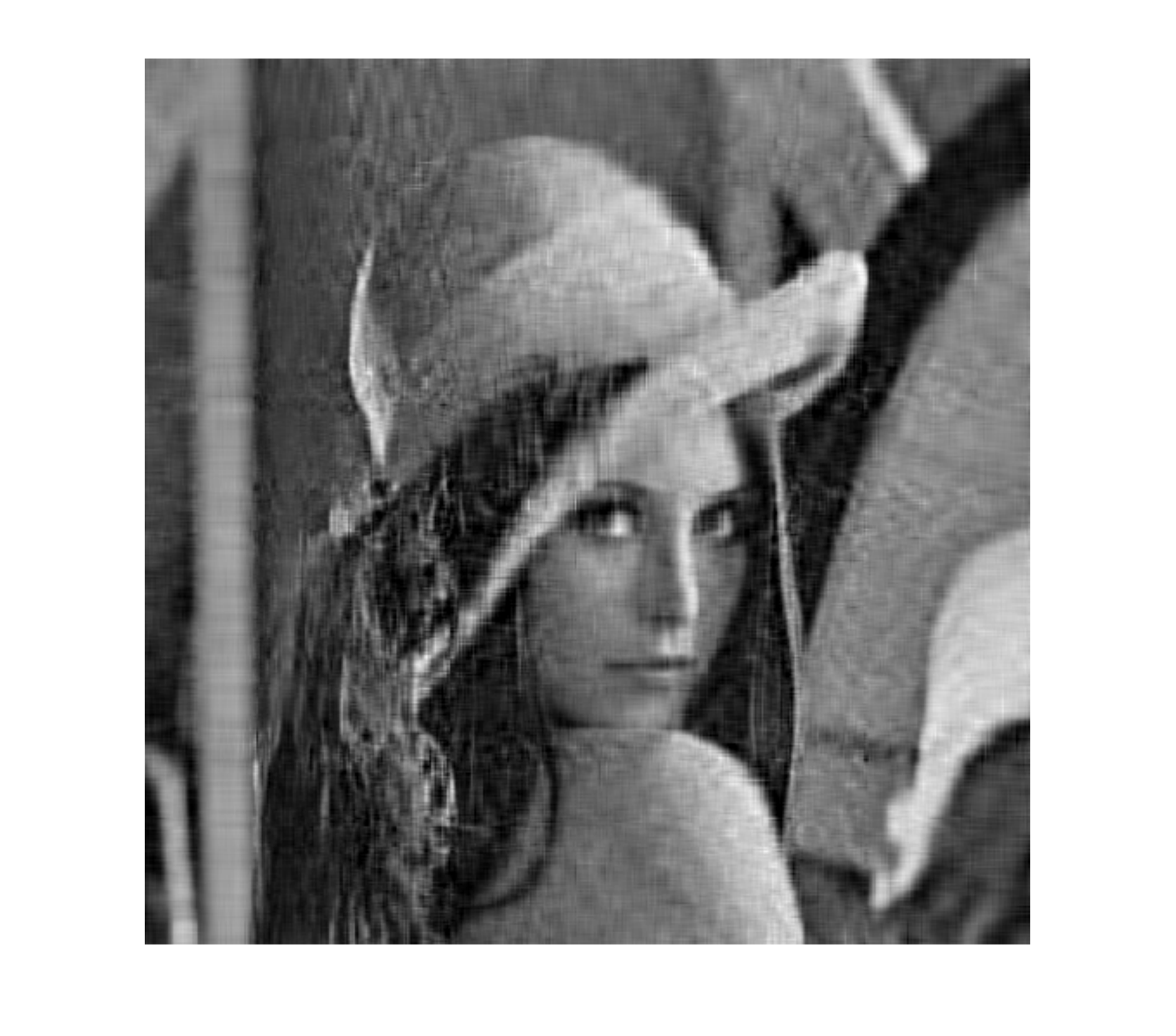} &
  \includegraphics[width=1.5in]{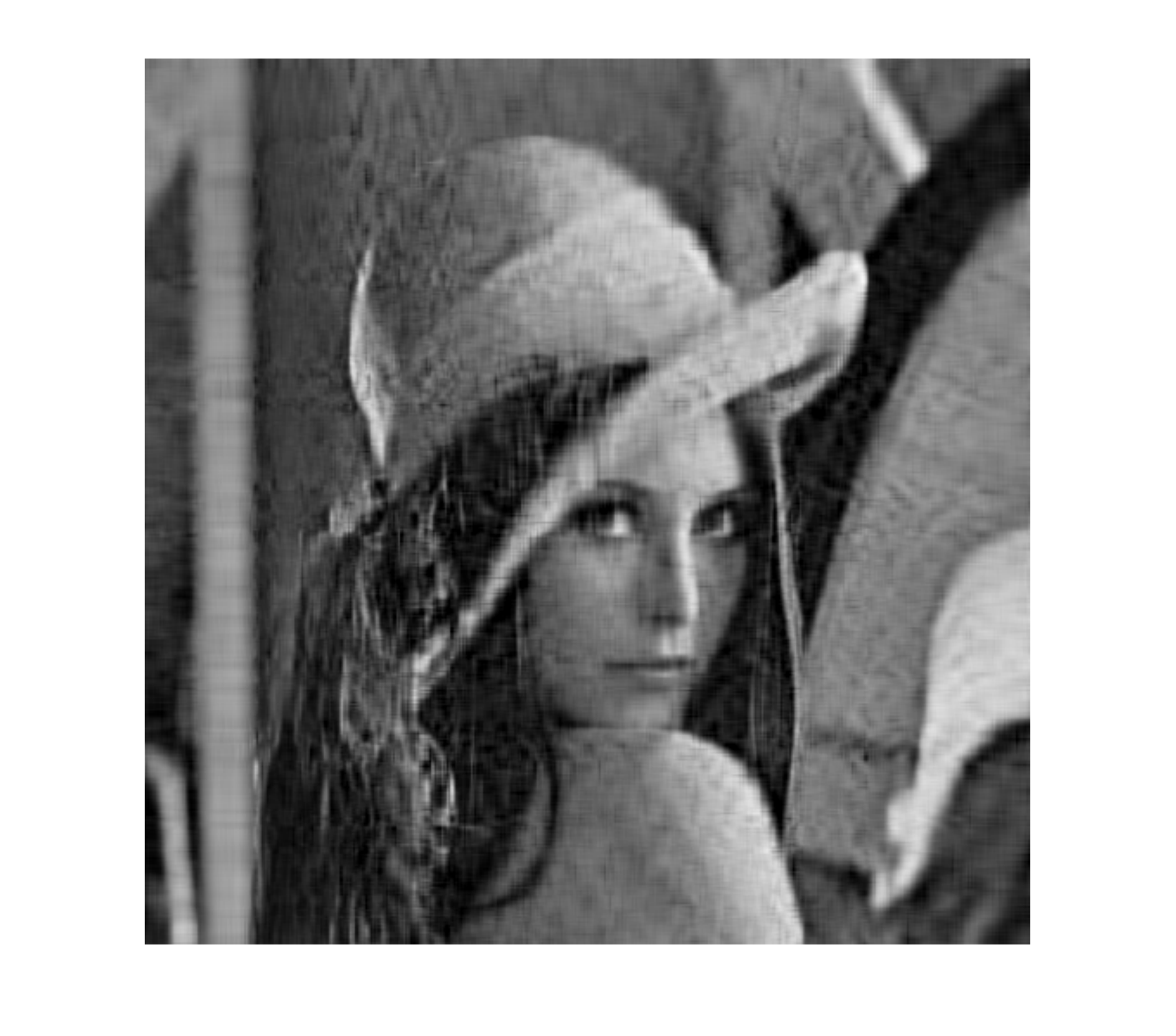}
  \\
 &\footnotesize $\|X(\mbox{rank})\|_F= 223.9999$ & \footnotesize$Error=13.1394$ & \footnotesize $Error=12.6303$
   \\ \hline
50 & \includegraphics[width=1.5in]{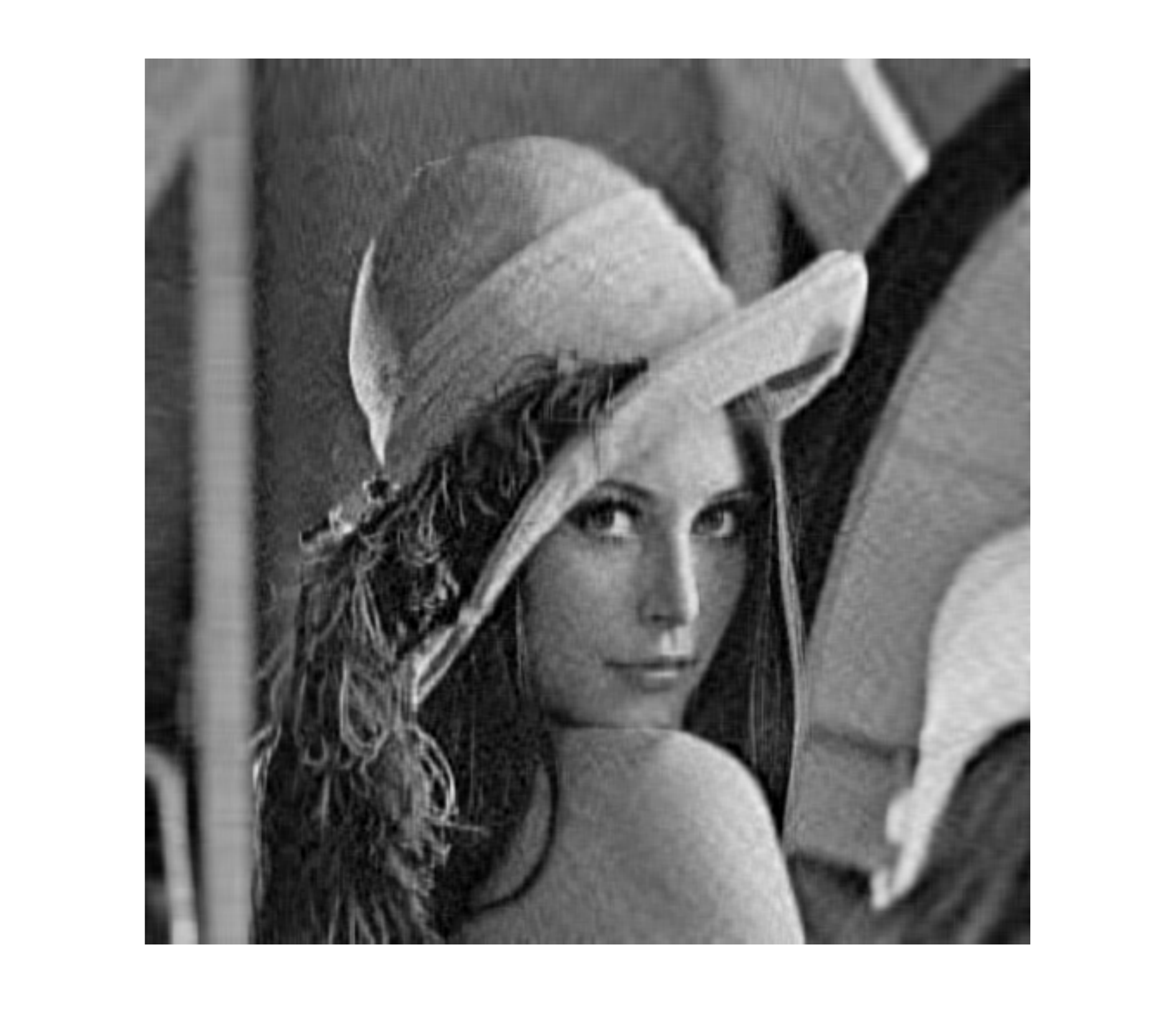} &
  \includegraphics[width=1.5in]{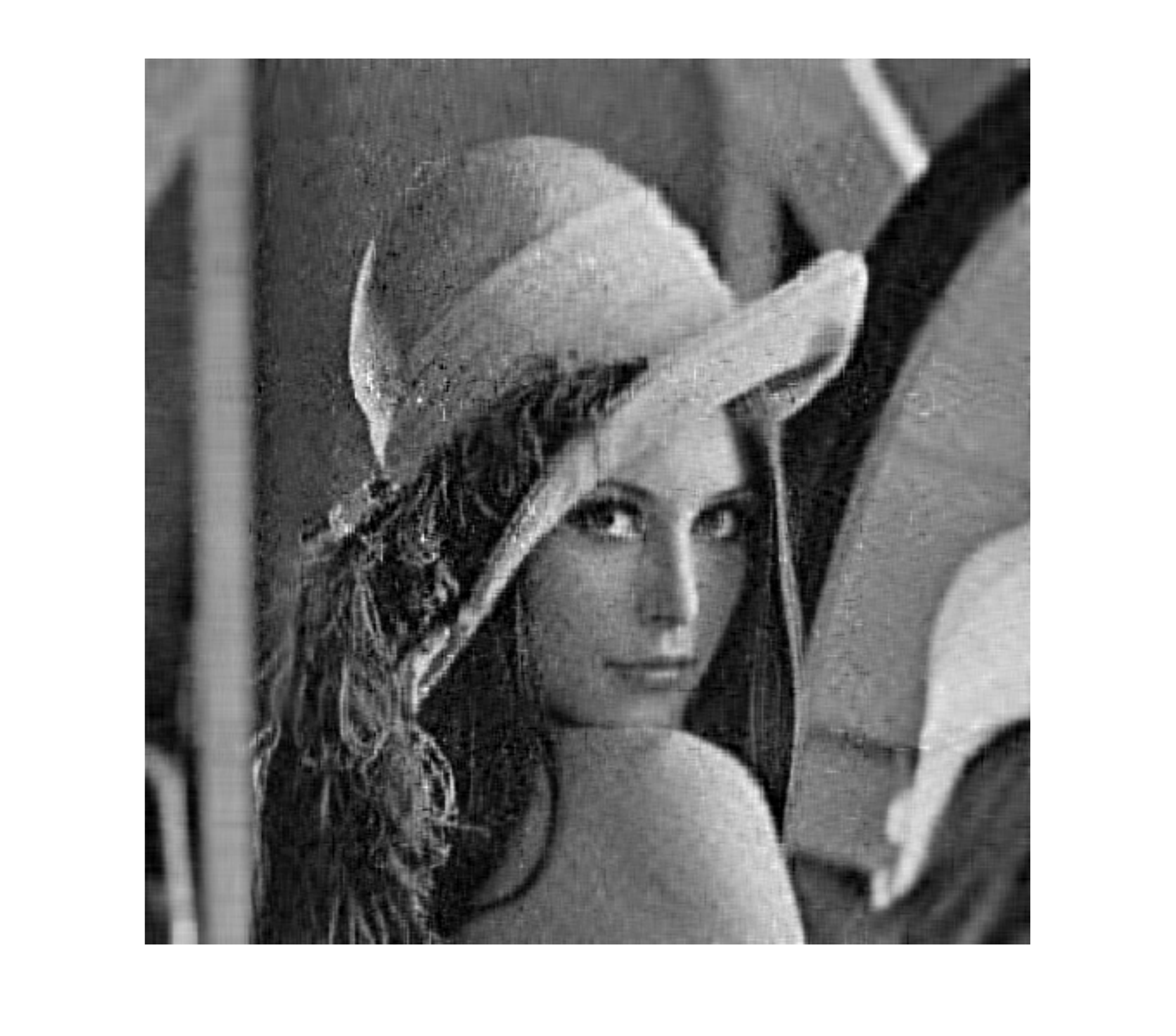} &
  \includegraphics[width=1.5in]{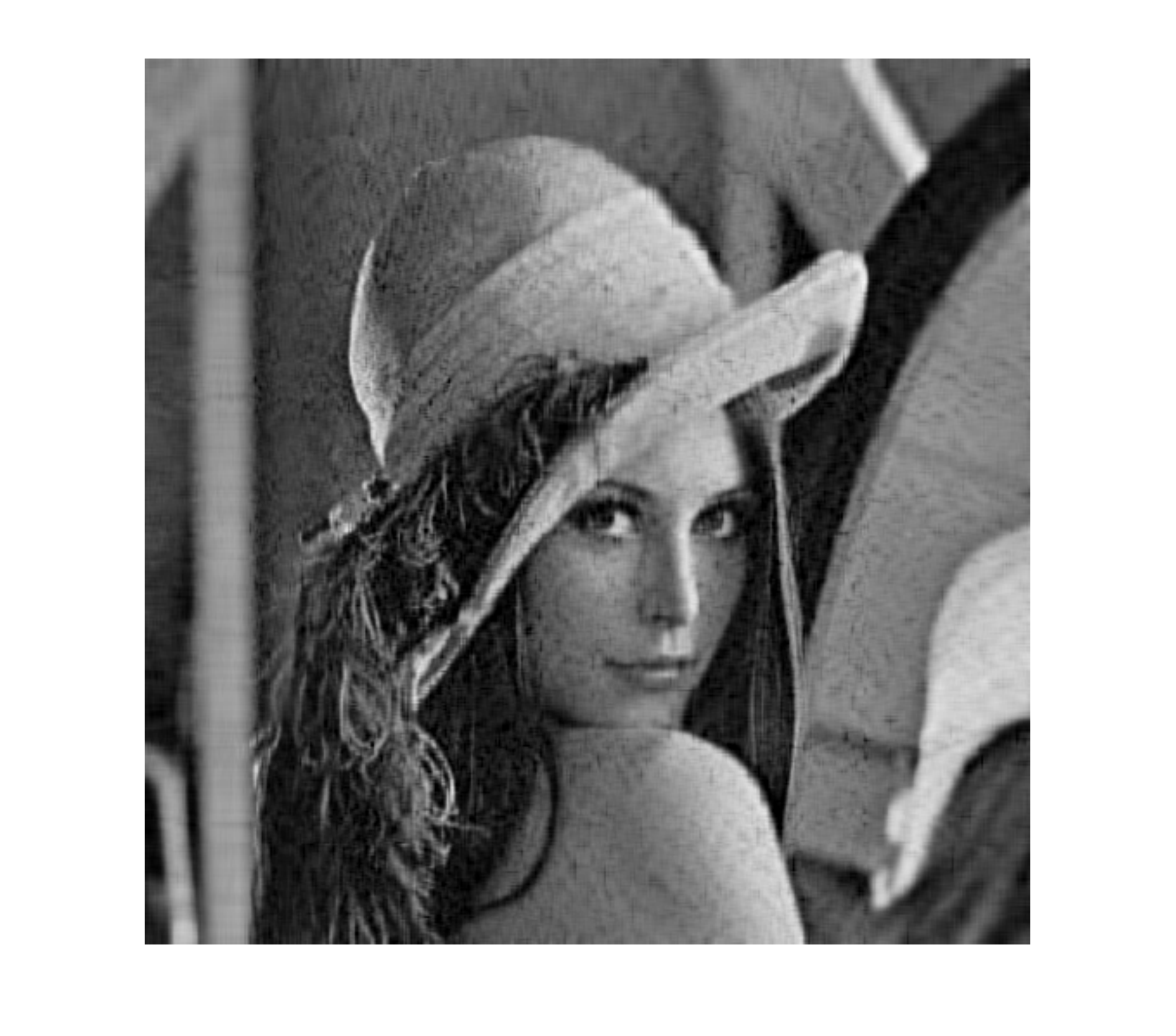}
  \\
   & \footnotesize$\|X(\mbox{rank})\|_F= 224.6876$ & \footnotesize$Error=18.2070$ &  \footnotesize$Error= 13.1859$

   \\ \hline
100 & \includegraphics[width=1.5in]{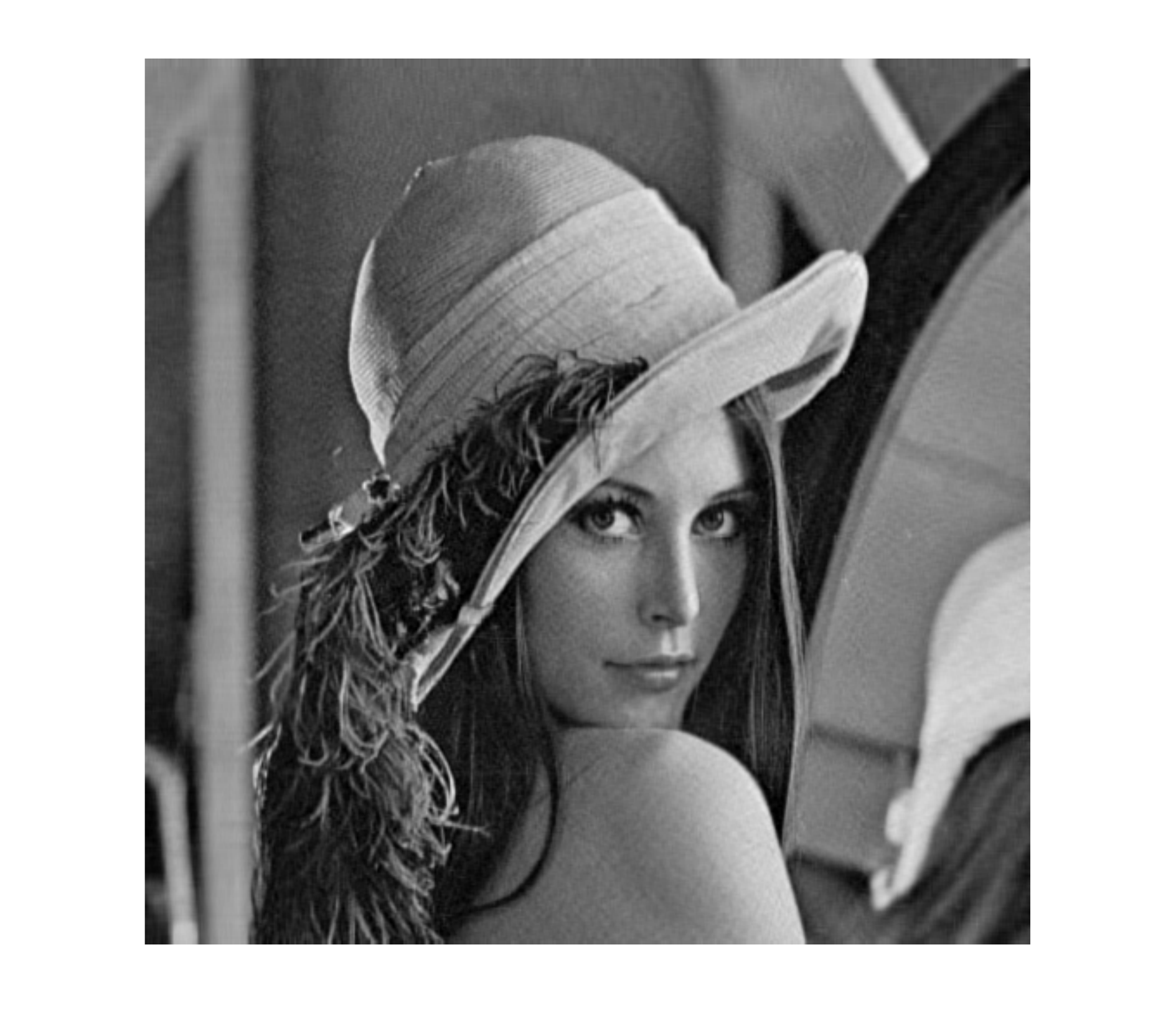} &
  \includegraphics[width=1.5in]{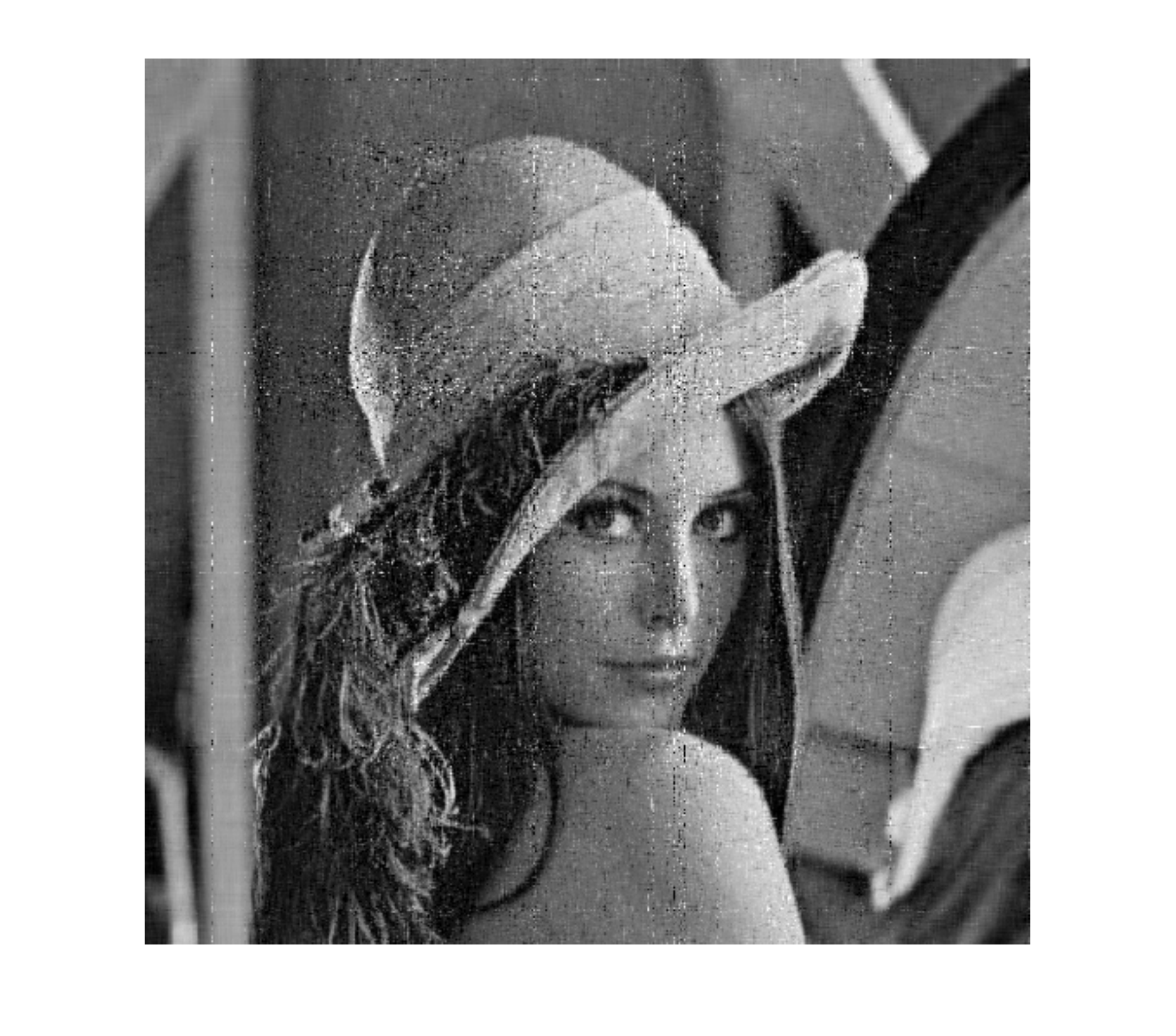} &
  \includegraphics[width=1.5in]{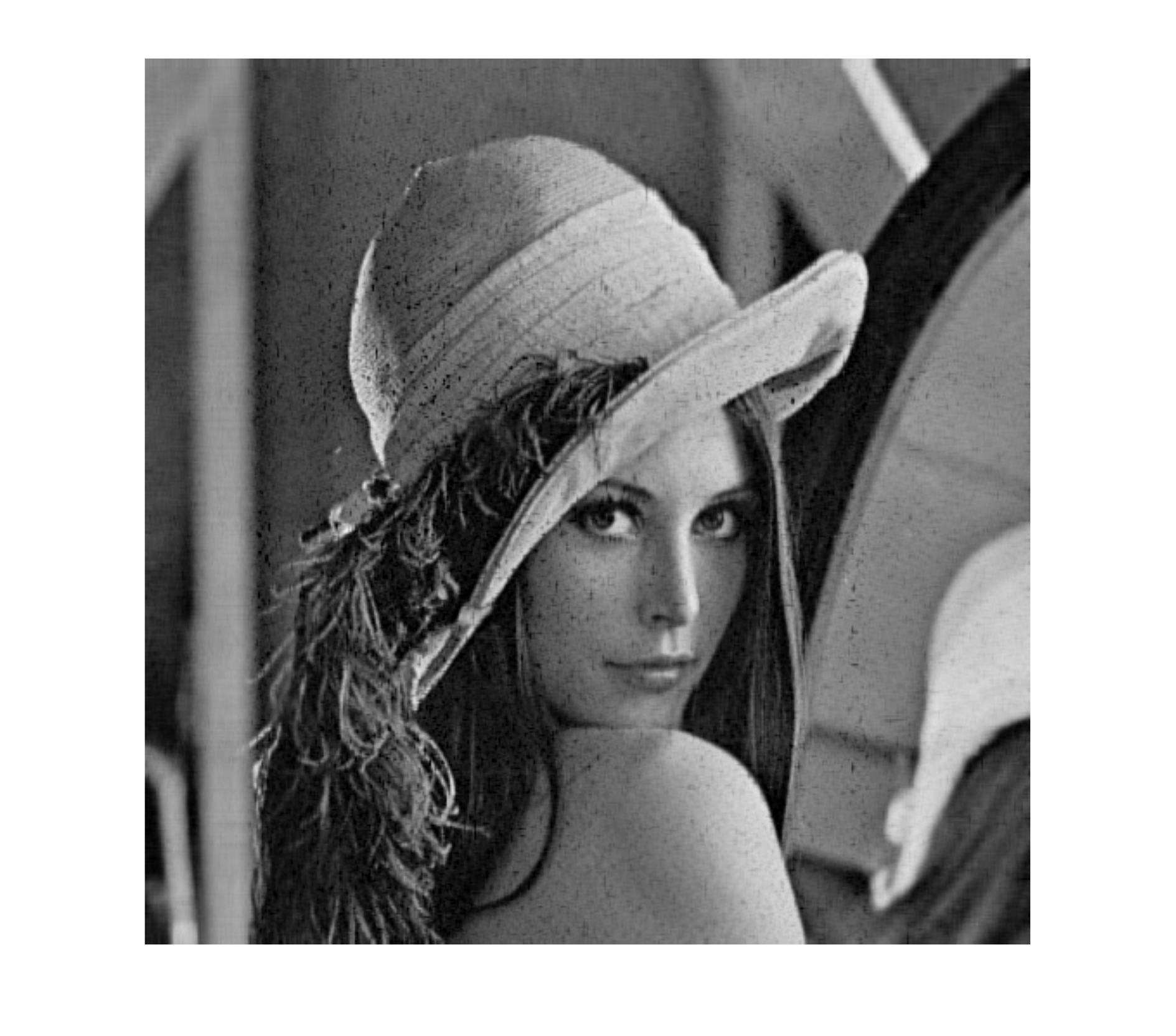}
  \\
    & \footnotesize $\|X(\mbox{rank})\|_F= 225.2117$ &\footnotesize $Error=39.1631$ &  \footnotesize$Error= 15.2551$

\normalsize
 \end{tabular}
 \caption{Adding obvious constraints can help to get better solution. Error is defined as $Error := \|X(rank)-X\|_F$.}
 \label{fig:EXP3}
\end{figure*}

\begin{landscape}
\begin{table*}[tp!]
\caption{Comparison with other solvers on the image recovery in terms of the peak signal-to-noise ratio (PSNR), citing the experiments of Wang et al.\ 
and adding results considering $0 \le Y_{i,j} \le 255$ under ``MACO''.} \label{table:ir}
\centering
\begin{small}
\begin{tabular}{ |l|c|c|c|c|c|c|c|c|c|}
\toprule 
Instance / Algo. & SVT & SVP & SoftImpute & LMaFit  & ADMiRA & JS & OR1MP & EOR1MP & MACO \\
\midrule
Barbara & {\bf 26.9635} & 25.2598 & 25.6073 & 25.9589 & 23.3528& 23.5322 & 26.5314 & 26.4413 & 23.8015\\
\hline
Cameraman & 25.6273 & 25.9444 & 26.7183 & 24.8956 & 26.7645 & 24.6238 & 27.8565  & 27.8283 & {\bf 28.9670}\\
\hline
Clown  & 28.5644  & 19.0919 & 26.9788 & 27.2748 & 25.7019& 25.2690 & 28.1963& 28.2052 & {\bf 29.0057}\\
\hline
Couple & 23.1765 & 23.7974 & 26.1033 & 25.8252 & 25.6260 & 24.4100 & 27.0707  & 27.0310 & {\bf 27.1824}\\
\hline
Crowd  & {\bf 26.9644} & 22.2959 & 25.4135 & 26.0662 & 24.0555 & 18.6562 & 26.0535 & 26.0510 & 26.1705\\
\hline
Girl  & 29.4688 & 27.5461 & 27.7180 & 27.4164 & 27.3640  & 26.1557 & 30.0878  & 30.0565 & {\bf 30.4110}\\
\hline
Goldhill & 28.3097 & 16.1256 & 27.1516 & 22.4485 & 26.5647 & 25.9706 & 28.5646 & 28.5101 & {\bf 28.6265}\\
\hline
Lenna & 28.1832  & 25.4586 & 26.7022 & 23.2003 & 26.2371 & 24.5056 & 28.0115 & 27.9643 & {\bf 28.3581}\\
\hline
Man  & {\bf 27.0223} & 25.3246 & 25.7912 & 25.7417 & 24.5223 & 23.3060 & 26.5829 & 26.5049&  26.5990\\
\hline
Peppers  &  25.7202 & 26.0223 & 26.8475 & 27.3663 &25.8934 & 24.0979 & 28.0781  & 28.0723 & {\bf 28.8469}\\
\bottomrule
\end{tabular}
\end{small}
\end{table*}
\end{landscape}

Further, we took a $50\times 50$ image and sampled randomly 50\% of pixels. (The image is the top-left corner of the Lenna image.)
Figure \ref{fig:EXP0} shows the original image $X$ and the best rank 10 approximation $X(10)$.
The solutions $X_\sE$,  $X_{\sE + \sB}$, $X_{\sE + \sA}$ and  $X_{\sE + \sB + \sA}$ were obtained by running Algorithm \ref{alg:SCDM}
for $3 \times 10^5$ serial iterations ($|\hatS|=1$), where $\sE$ contains the observed pixels and $\sB$ and $\sA$ contains all other pixels.
We have used $X^\sL = {\bf 0}$ and $X^{\sU} = {\bf 1}$.
The result  again suggest that adding simple and obvious constrains leads to better low rank reconstruction and helps to keep reconstructed elements of matrix in expected bounds.

%
%
%

 \begin{figure*}[tp!]
 \centering \footnotesize
 \begin{tabular}{c|c|c}
   $X$ & 
   $X(10) $ &
   $X_\sE$     \\
 \includegraphics[width=1.5in]{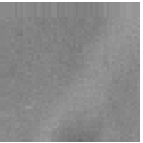}  &
 \includegraphics[width=1.5in]{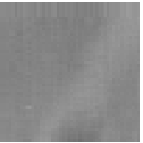}  &
 \includegraphics[width=1.5in]{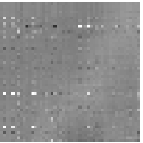}  \\
$\|X\|_F =  26.63$ &
$\|X(10)\|_F =  26.63$ &
$RE =  0.1031$ \\
   $X_{\sE + \sB}$   &
   $X_{\sE + \sA}$   &
   $X_{\sE + \sB + \sA}$\\
 \includegraphics[width=1.5in]{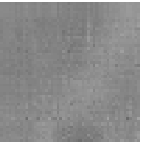}  &
 \includegraphics[width=1.5in]{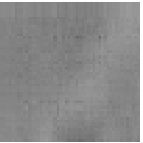}  &
 \includegraphics[width=1.5in]{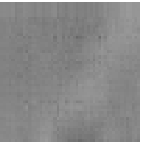} \\
$RE =  0.0357$ &
$RE =  0.0262$ &
$RE =  0.0262$\\    
 \end{tabular}
\normalsize
 \caption{Original $50\times 50$ image, the best rank 10 approximation and reconstruction using Algorithm \ref{alg:SCDM} with different settings.
 The $RE$ is a relative error defined as $RE(X_\cdot) = \|X_\cdot - X(10)\|_F / \|X(10)\|$.}
 \label{fig:EXP0}
\end{figure*}

\subsection{The Run-Time}

Finally, in order to illustrate the run-time and efficiency of parallelization of Algorithm \ref{alg:SCDM},
Figure \ref{fig:RMSE} (right) 
presents the evolution of RMSE over time on the well-known $480,189 \times 17,770$ 
matrix of rank 20.
There is an almost linear speed-up visible from 1 to 4 cores and marginally worse speed-up between 4 and 8 cores.
Considering that most other algorithms proposed in the literature cannot cope with instances of this size, we 
cannot compare the performance directly to 
SVT \cite{candes2009exact}, SVP \cite{jain2010}, SoftImpute \cite{mazumder2010spectral}, LMaFit \cite{goldfarb2009solving}, ADMiRA, \cite{Lee2010}, JS \cite{jaggi2010simple}, 
OR1MP \cite{wang2014rank}, EOR1MP \cite{wang2014rank}, and similar.
We can, however, compare the run-time on the $512 \times 512$ instances, 
detailed in Table~\ref{table:ir}.

\section{Conclusions}

We have studied the matrix completion problem under interval uncertainty and an efficient algorithm, which converges to stationary points of the NP-Hard, non-convex optimisation problem,
without ever trying to approximate the spectrum of the matrix.
In our computational experiments, we have shown that even the seemingly most trivial inequality constraints are useful in a number of applications.
This opens numerous avenues for further research:

\begin{itemize}
\item  
Forecasting with Side Information: 
A related application comes from the forecasting of seasonal data, e.g. sales.
Let us assume that in process $\left\{X_t\right\}$, one knows $k + 1 = \tau$ such that
$ F_{X}(x_{t_1+\tau} ,\ldots, x_{t_k+\tau}) = F_{X}(x_{t_1},\ldots, x_{t_k})$
for the cumulative distribution function $F_{X}(x_{t_1 + \tau}, \ldots, x_{t_k + \tau})$ of the joint distribution of $\left\{X_t\right\}$ at times $t_1 + \tau, \ldots, t_k + \tau$.
One can then formulate the forecasting into the future as a matrix completion problem,
where there the historical datum at time $t$ is at row $\lfloor t / \tau \rfloor$, column $t \mod k$  specified by an equality or a pair of inequalities,
and where inequalities represent side information. 
For example in sales forecasts,
one often has bookings for many months in advance and knows that the sales for the
respective months will not be less than the bookings taken.

\item Non-negative matrix factorization: The coordinate descent algorithm for the problem \eqref{eq:NONCREF} is easy to extend, e.g., toward non-negative factorization. It is sufficient to modify
lines 7 and 13 in Algorithm \ref{alg:SCDM} as follows:
$L_{i,\hat r} = \max \{0, L_{i,\hat r} + \delta_{i,\hat r} \}$,
$R_{\hat r,j} = \max \{0, R_{\hat r,j} + \delta_{\hat r,j} \}$.
One could consider extensions beyond box constraints on the individual elements as well.

\item Auto-tuning $\mu$: If we have some \emph{a priori} bound on the largest eigenvalue of the matrix to reconstruct, let us denote it $\zeta$,
then we can modify
lines 7 and 13 in Algorithm \ref{alg:SCDM} as follows
$L_{i,\hat r} = \max\{ \min \{\zeta, L_{i,\hat r} + \delta_{i,\hat r} \}, -\zeta \}$,
$R_{\hat r,j} = \max\{ \min  \{0, R_{\hat r,j} + \delta_{\hat r,j} \}, -\zeta \}$.
\end{itemize}

We would be delighted to share our code with other researchers interested in these and related problems. Currently, the code is available from \url{http://optml.github.io/ac-dc/}. Should it
become unavailable, for any reason, we encourage researchers to contact us.

\clearpage

\bibliographystyle{plain}
 \bibliography{completion} 

\clearpage

\paragraph*{Acknowledgements} 
The authors are grateful for the reviews, which have helped them to improve both the presentation and contents of the paper.
In addition, the first author acknowledges funding from the European Union Horizon 2020 Programme (Horizon2020/2014-2020), under grant agreement number 688380. 
The second author would like to acknowledge support from the EPSRC Grant EP/K02325X/1, {\em Accelerated Coordinate Descent Methods for Big Data Optimization.}

\end{document}